\documentclass[11pt, twoside]{article}
\usepackage{amsfonts}
\usepackage{ulem}

\usepackage{amssymb}
\usepackage{amsmath}
\usepackage{amsthm}
\usepackage{xcolor}
\usepackage{mathrsfs}

\usepackage{verbatim}

\allowdisplaybreaks

\allowdisplaybreaks

\def\rr{{\mathbb R}}

\def\rn{{{\rr}^n}}

\def\fz{\infty}

\def\cm{{\mathcal M}}

\def\loc{{\rm{\,loc\,}}}

\def\lz{\lambda}

\def\hs{\hspace{0.3cm}}

\def\r{\right}
\def\lf{\left}

\def\bint{{\ifinner\rlap{\bf\kern.30em--}
\int\else\rlap{\bf\kern.35em--}\int\fi}\ignorespaces}

\def\sbint{{\ifinner\rlap{\bf\kern.32em--}
\hspace{0.078cm}\int\else\rlap{\bf\kern.45em--}\int\fi}\ignorespaces}

\newtheorem{theorem}{Theorem}[section]
\newtheorem{lemma}[theorem]{Lemma}
\newtheorem{corollary}[theorem]{Corollary}

\theoremstyle{definition}

\newtheorem{remark}[theorem]{Remark}
\newtheorem{property}[theorem]{Property}
\newtheorem{definition}[theorem]{Definition}
\numberwithin{equation}{section}

\numberwithin{equation}{section}

\numberwithin{equation}{section}

\begin{document}

\arraycolsep=1pt

\title{\Large\bf Characterization of Lipschitz Space via the Commutators of Fractional Maximal Functions on Variable Lebesgue Spaces\footnotetext{\hspace{-0.35cm} {\it 2020
Mathematics Subject Classification}. {42B35, 42B25, %(函数空间, 极大函数, Littlewood-Paley theory)
47B47, %(Commutators, derivations, elementary operators, etc.)
26A16.} %Lipschitz (H¨older) classes
\endgraf{\it Key words and phrases.} variable Lebesgue space, Lipschitz space,  maximal function, fractional maximal function, commutator.
\endgraf $^\ast$\,Corresponding author.
}}
\author{Xuechun Yang, Zhenzhen Yang and Baode Li$^\ast$}
\date{ }
\maketitle

\vspace{-0.8cm}

\begin{center}
\begin{minipage}{13cm}\small
{\noindent{\bf Abstract.}
We obtain some new characterizations of a variable version of Lipschitz spaces in terms of the boundedness of commutators of sharp maximal functions, fractional maximal functions or fractional maximal commutators in the context of the variable Lebesgue spaces, where the symbols of the commutators belong to the variable Lipschitz space.
A useful tool is that a symbol $b$ belongs a variable Lipschitz space of pointwise type if and only if $b$ belongs to a variable Lipschitz space of integral type.}
\end{minipage}
\end{center}
%%%%%%%%%%%%%%%%%%%%%%%%%%%%%%%%%%%%%%%%%%%%%%%%%%%%%%%%%%%%%%%%%%%%%%%%%
%%%%%%%%%%%%%%%%%%%%%%%%%%%  1.  Preliminaries  %%%%%%%%%%%%%%%%%%%%%%%%%
%%%%%%%%%%%%%%%%%%%%%%%%%%%%%%%%%%%%%%%%%%%%%%%%%%%%%%%%%%%%%%%%%%%%%%%%%

\section{Introduction\label{s1}}
It is well known that the commutators of a great variety of operators appearing in harmonic analysis are intimately related to the regularity properties of the solutions of certain partial differential equations, see for example \cite{bmr, bcm96, dr93, cr03,x07,pr17,tyyz,x17}.
%It is known that the Lipschitz space coincides with some Morrey-Companato space and can be characterized by mean oscillation, which is due to DeVore and Sharpley \cite{ds84} and Janson, Taibleson and Weiss \cite{jtw83}.
In 1976, a first result in this direction was established
by Coifman, Rochberg and Weiss in \cite{crw76}, where the authors proved that the space of the bounded mean oscillation functions, $BMO$, is characterized by the boundedness of the commutator of singular integral operators with symbol in $BMO$. Later, in 1995, Paluszy\'{n}ski \cite{p95} proved that Lipschitz spaces can be characterized by means of the boundedness between Lebesgue spaces of the commutators of fractional integral operators with symbols belonging to Lipschitz spaces. In 2017, Pradolini and Ramos \cite{pr17} further characterized a variable version of Lipschitz space via the boundedness between variable Lebesgue space of commutators by  Calder\'on-Zygmund operators with symbols belonging to variable
Lebesgue spaces. For more details for variable Lipschitz spaces, see \cite{rsv13}. Here
it is worth pointing out that the variable Lipschitz space can be reduced to the classical Lipschitz space and $BMO(\mathbb{R}^n)$ space as special cases (see Remark \ref{r2.5} below).
%In 1978, Janson \cite{j78} gave a characterization in terms of the boundedness of the commutators of singular integral operators with Lipschitz function symbols in this class.

%Let $T$ be the Calder\'on-Zygmund operator, the commutator $[b,\,T]$ generated by $T$ and a suitable function $b$ is given by
%\begin{equation}\label{e1.1}
%[b,\,T](f)(x)=T(b(x)-b)f(x)=b(x)T(f)(x)-T(bf)(x).
%\end{equation}

Besides, Lipschitz spaces can be also characterized via boundedness between Lebesgue spaces of the commutators of fractional maximal functions with symbols belonging to Lipschitz spaces. Let us begin with some definitions. As usual, a cube $Q\subset\mathbb{R}^n$ always means its sides parallel to the coordinate axes.
Denote by $|Q|$ the Lebesgue measure of $Q$ and $\chi_Q$ the characteristic function of $Q$.
For $f\in L^1_{\loc}(\mathbb{R}^n)$, we write $f_Q:=|Q|^{-1}\int_Qf(x)\,dx$. The {\it sharp maximal function} $\mathcal{M}^{\sharp}$ and the {\it fractional maximal function} $\mathcal{M}_\alpha$ are, respectively, defined by
\begin{equation*}
\mathcal{M}^{\sharp}f(x):=\sup_{Q\ni x}\frac{1}{|Q|}\int_Q |f(y)-f_Q|\,dy
\end{equation*}
and
\begin{equation*}
\mathcal{M}_\alpha f(x):=\sup_{Q\ni x}\frac{1}{|Q|^{1-\frac{\alpha}{n}}}\int_Q|f(y)|\,dy, \quad 0\le\alpha<n,
\end{equation*}
where the supremum is taken over all cubes $Q\subset\mathbb{R}^n$ containing $x$. When $\alpha=0$, simply
denote $\mathcal{M}:=\mathcal{M}_0$, which is exactly the {\it Hardy-Littlewood maximal function}.

The {\it fractional commutator of $\mathcal{M}_\alpha (0\leq \alpha<n)$} with a locally integrable function $b$ is defined by
\begin{equation*}
\mathcal{M}_{\alpha,\,b}(f)(x):=\sup_{Q\ni x}\frac{1}{|Q|^{1-\frac{\alpha}{n}}}\int_Q|b(x)-b(y)||f(y)|\,dy,
\end{equation*}
The {\it (nonlinear) commutators of $\mathcal{M}^{\sharp}$ and $\mathcal{M}_\alpha$$(0\leq \alpha<n)$} with
a locally integrable function $b$ are, respectively, defined by
\begin{equation}\label{e2.08}
[b,\,\mathcal{M}^{\sharp}](f):=b\mathcal{M}^{\sharp}(f)-\mathcal{M}^{\sharp}(bf)
\end{equation}
and
\begin{equation}\label{e2.10}
[b,\,\mathcal{M}_\alpha](f)(x):=b(x)\mathcal{M}_\alpha(f)(x)-\mathcal{M}_\alpha(bf)(x),
\end{equation}
where the supremum is taken over all cubes $Q\subset\mathbb{R}^n$ containing $x$. When $\alpha=0$, simply denote  $\mathcal{M}_b := \mathcal{M}_{0,\,b}$ and $[b,\,\mathcal{M}] := [b,\,\mathcal{M}_0]$.
We call $[b,\,\mathcal{M}_\alpha]$ the nonlinear commutator because it is not even a sublinear operator,
although the classical commutator $[b,\,T]$ of linear operator $T$ is a linear one. We would like to remark that the nonlinear
commutator $[b,\,\mathcal{M}_\alpha]$ and the maximal commutator $\mathcal{M}_{\alpha,\,b}$ essentially differ from each other.
For example, $\mathcal{M}_{\alpha,\,b}$ is positive and sublinear, but $[b,\,\mathcal{M}_\alpha]$ is neither positive nor sublinear.
%In 1990, by using the real interpolation techniques, Milman and Schonbek \cite{ms} obtained
%a commutator result, by which they obtained the $L^p$-boundedness of $[b,\,M]$ and $[b,\,M_\alpha](0<\alpha<n)$ when $b\in BMO(\mathbb{R}^n$) and $b\geq0$.

In 2000, Bastero, Milman and Ruiz \cite{bmr} obtained that $b\in BMO(\mathbb{R}^n)$ (i.e. $\cm^\sharp(b)\in L^\infty(\rn)$) and
$b^-:=-\min\{b,\,0\}\in L^\infty(\rn)$ if and only if commutators $[b,\,\mathcal{M}]$ or $[b,\,\cm^\sharp]$ is bounded on $L^p(\mathbb{R}^n)\,(1<p<\infty)$
%In 2017, Zhang \cite{z17} showed the boundedness of commutators $[b,\,\cm]$ (resp. $\cm_b$) in $L^p(\mathbb{R}^n)\,(1<p<\infty)$ if and only if $b\ge0$ (resp. $b$) belongs to a Lipschitz space.
In 2019, Zhang \cite{z19}, Si and Wu \cite{zsw19} further obtained that $b\ge0$ (resp. $b$) belongs to Lipschitz space if and only if the commutators $[b,\,\cm]$, $[b,\,\cm^\sharp]$ or $[b,\, \cm_\alpha]$ (resp. $\cm_b$ or $\cm_{\alpha,\,b}$) is bounded on variable Lebesgue spaces.
%On the other hand, in 2017, Pradolini and Ramos \cite{pr17} obtained the boundedness of commutators by  Calder\'on-Zygmund operators and locally integrable function $b$ in variable Lebesgue spaces if and only if
%$b$ belongs to a variable version of Lipschitz space. For more details for variable Lipschitz spaces, see \cite{rsv13}. Here
%it is worth pointing out that the variable Lipschitz space can be reduced to the classical Lipschitz space and $BMO(\mathbb{R}^n)$ space as special cases (see Remark \ref{r2.5} below).

Inspired by Pradolini et al. \cite{pr17} and Zhang et al. \cite{z19, zsw19}, a natural question is whether
variable Lipschitz space can be also characterized via the boundedness between variable Lebesgue spaces of the commutators of the fractional maximal functions with symbols belonging to variable Lipschitz space? This paper gives an affirmative answer. Precisely, we obtain that $b\ge 0$ (resp. $b$) belongs to variable Lipschitz space if and only if the commutators of $[b,\,\mathcal{M}^{\sharp}]$, $[b,\,\mathcal{M}_\alpha]$, or $[b,\,\mathcal{M}]$ (resp. $\mathcal{M}_{\alpha,\,b}$ or $\mathcal{M}_b$) are bounded on variable Lebesgue spaces. (see Theorems \ref{t3.3} \ref{t3.4}, \ref{t3.5} and Corollaries \ref{t3.1}, \ref{t3.2} below).

To be precise, this article is organized as follows.

In Section \ref{s2}, we first recall notation, definitions and properties of variable Lebesgue spaces and variable Lipschitz spaces. Then we prove the equivalence of variable Lipschitz space of pointwise type and variable Lipschitz space of integral type. As applications, the boundedness results of commutators $[b,\,\mathcal{M}^{\sharp}]$, $\mathcal{M}_{\alpha,\,b}$, $[b,\,\mathcal{M}_\alpha]$, $\mathcal{M}_b$ and $[b,\,\mathcal{M}]$ are obtained in  Section \ref{s3} (see Theorems \ref{t3.3}, \ref{t3.4} and \ref{t3.5}, Corollaries \ref{t3.1} and \ref{t3.2} below). The main idea is to improve the relevant proofs on the Lipschitz space (\cite[Theorems 1.3]{z19} and \cite[Theorems 1.3 and 1.1]{zsw19}) to the variable Lipschitz space. Here, by referring to \cite{ks03,ms12}, the boundedness of variable fractional maximal functions on variable Lebesgue spaces plays a key role (see Remark \ref{r2.1} below).

\section*{Notations}

\ \,\ \ $-C:$ a positive constant which is independent of the main parameters;

$-\mathbb{C}:$ the set of complex numbers;

$-Q:$ a cube in $\mathbb{R}^n$ with sides parallel to the coordinate axes;

$-|Q|:$ the Lebesgue measure of $Q$;

$-\chi_Q:$ the characteristic function of $Q$.

%%%%%%%%%%%%%%%%%%%%%%%%%%%%%%%%%%%%%%%%%%%%%%%%%%%%%%%%%%%%%%%%%%%%%

%%%%%%%%%%%%%%%%%%%%%%%  section 2  %%%%%%%%%%%%%%%%%%%%%%%%%%%%%%%%%

%%%%%%%%%%%%%%%%%%%%%%%%%%%%%%%%%%%%%%%%%%%%%%%%%%%%%%%%%%%%%%%%%%%%%

\section{Variable Lipschitz spaces \label{s2}}
Let us recall two kinds of variable Lipschitz spaces and prove their equivalence.

For any measurable function $p(\cdot):\mathbb{R}^n\rightarrow(1,\,\infty)$, let
\begin{equation*}
p_{-}:=\mathop\mathrm{\,inf\,}_{x\in \mathbb{R}^n}p(x), \qquad p_{+}:=\mathop\mathrm{\,sup\,}_{x\in \mathbb{R}^n}p(x).
\end{equation*}
Define $\mathcal{P}(\mathbb{R}^{n})$ the set of all measurable functions
$p(\cdot)$ satisfying $1<p_{-}\leq p_{+}<\infty$.

For any measurable function $p(\cdot)\in \mathcal{P}(\mathbb{R}^{n})$,
the {\it variable Lebesgue space} $L^{p(\cdot)}(\mathbb{R}^{n})$ denotes the set of
measurable functions $f$ on $\mathbb{R}^n$ such that, for some $\lambda>0$,
\begin{equation*}
\varrho_{p(\cdot)}(f/\lambda):=\int_{\mathbb{R}^n}\lf(\frac{|f(x)|}{\lambda}\r)^{p(x)}\,dx<\infty.
\end{equation*}
This set becomes a Banach function space when equipped with the Luxemburg-Nakano norm
\begin{equation*}
\|f\|_{L^{p(\cdot)}(\mathbb{R}^n)}:=\inf\lf\{ \lz\in(0,\,\fz):\ \int_{\mathbb{R}^n}\lf(\frac{|f(x)|}
{\lambda}\r)^{p(x)}\,dx\le 1\r\}.
\end{equation*}

Let $\mathcal{P}^{\log}(\mathbb{R}^{n})$ be the set of all functions $p(\cdot)\in \mathcal{P}(\mathbb{R}^{n})$
satisfying the {\it globally log-H\"{o}lder continuous condition}, namely, there exist
$C_{\log}(p)$, $C_{\infty}\in(0, \infty)$ and $p_{\infty}\in\mathbb{R}$ such that,
for any $x, y\in\mathbb{R}^{n}$,
\begin{equation}\label{e2.1}
|p(x)-p(y)|\leq\frac{C_{\log}(p)}{\ln (e+\frac{1}{|x-y|})}
\end{equation}
and
\begin{equation}\label{e2.2}
|p(x)-p_{\infty}|\leq\frac{C_{\infty}}{\ln (e+|x|)}.
\end{equation}
It is easy to note that the inequality \eqref{e2.2} implies that
$\lim_{|x|\rightarrow\infty}p(x)=p_{\infty}$.

The class $\mathcal{P}^{\log}(\rn)$ is introduced by Cruz-Uribe, Fiorenza and Neugebauer \cite{cfn03} to treat maximal functions in spaces with variable exponent on $\rn$; see the following property.

\begin{property}\label{p2.11}\cite[Theorem 1.5]{cfn03}
If $p(\cdot)\in \mathcal{P}^{\log}(\rn)$, then the maximal operator $\mathcal{M}$ is bounded on $\mathop L\nolimits^{p\left(  \cdot  \right)} \left( {\mathop \mathbb{R}\nolimits^n } \right)$.
\end{property}

Given $p(\cdot)\in \mathcal{P}(\mathbb{R}^{n})$, define the conjugate exponent
$p'(\cdot)$ by the equation
$$\frac{1}{p(\cdot)}+\frac{1}{p'(\cdot)}=1.$$

It is easy to check that if $p(\cdot)\in \mathcal{P}(\mathbb{R}^n)$
then $p'(\cdot)\in \mathcal{P}(\mathbb{R}^n)$. The following are some basic properties of variable Lebesgue spaces.

\begin{property}\label{p2.2}
\item[{(i)}]\cite[Theorem 2.26]{cf13} (H\"older's inequality) Let $p(\cdot)\in \mathcal{P}(\mathbb{R}^{n})$. Then there exists a positive constant $C$ such that, for all $f\in L^{p(\cdot)}(\mathbb{R}^n)$ and
$g\in L^{p^\prime(\cdot)}(\mathbb{R}^n)$,\\
$$\int_{\mathbb{R}^n}|f(x)g(x)|\,dx\leq \|f\|_{L^{p(\cdot)}(\mathbb{R}^n)}\|g\|_{L^{p^\prime(\cdot)}(\mathbb{R}^n)}.$$
\item[{(ii)}]\cite[Corollary 2.28]{cf13} Let $p(\cdot),\,p_1(\cdot),\,p_2(\cdot)\in \mathcal{P}(\mathbb{R}^{n})$ and $1/p(\cdot)=1/p_1(\cdot)+1/p_2(\cdot)$. Then there exists a positive constant $C$ such that, for all $f\in L^{p_1(\cdot)}(\mathbb{R}^n)$ and $g\in L^{p_2(\cdot)}(\mathbb{R}^n)$,\\
$$\|fg\|_{L^{p(\cdot)}(\mathbb{R}^n)}\leq C\|f\|_{L^{p_1(\cdot)}(\mathbb{R}^n)}\|g\|_{L^{p_2(\cdot)}(\mathbb{R}^n)}.$$

\item[{(iii)}]\cite[Lemma 2.3]{cw14} Let $p(\cdot)\in \mathcal{P}(\mathbb{R}^{n})$ and $s\in(0, \infty)$. Then for $f\in L^{p(\cdot)}(\mathbb{R}^{n})$,
$$\||f|^{s}\|_{L^{p(\cdot)}(\mathbb{R}^{n})}=\|f\|^{s}_{L^{sp(\cdot)}(\mathbb{R}^{n})}.$$
\item[{(iv)}]\cite[P. 21]{cf13} For any $\lambda\in\mathbb{C}$ and $f,\ g\in L^{p(\cdot)}(\mathbb{R}^{n})$,
$$\|\lambda f\|_{L^{p(\cdot)}(\mathbb{R}^{n})}= |\lambda|\|f\|_{L^{p(\cdot)}(\mathbb{R}^{n})}.$$
\end{property}

Next, let us recall two kinds of definitions of variable Lipschitz spaces.

\begin{definition}\label{d2.4}
 Let $p(\cdot)\in \mathcal{P}(\mathbb{R}^{n})$, $0\le\delta(\cdot)<n$, $1<\beta\leq p_-$ and $ \delta(\cdot)/n=1/\beta-1/p(\cdot)$.

\begin{enumerate}
\item[(i)] We say a function $b$ belongs to the {\it variable Lipschitz space $\mathbb{L}(\delta(\cdot))$ of pointwise type} if there exists a constant $C$ such that, for all $x,y\in \mathbb{R}^n$,
$$|b(x)-b(y)|\leq C|x-y|^{\delta(x)}.$$
The smallest such constant $C$ is called the $\mathbb{L}(\delta(\cdot))$ norm of $b$ and is denoted by $\|b\|_{\mathbb{L}(\delta(\cdot))}$.
\item[(ii)] \cite[Definition 1.7]{rsv13} We say a function $b$ belongs to the {\it variable Lipschitz space of integral type} if there exists a constant $C$ such that
$$\|b\|_{\mathbb{\widetilde L}(\delta(\cdot))}=\sup_Q\frac{1}{|Q|^{\frac{1}
{\beta}}\|\chi_Q\|_{L^{p^\prime(\cdot)}(\mathbb{R}^n)}}\int_Q|b(x)-b_Q|\, dx<\infty.$$
The smallest such constant $C$ is called the $\mathbb{\widetilde L}(\delta(\cdot))$ norm of $b$ and is denoted by $\|b\|_{\mathbb{\widetilde L}(\delta(\cdot))}$.
\end{enumerate}
\end{definition}

\begin{remark}\label{r2.5}
When $p(x)$ is equal to a constant $p$, the variable Lipschitz space $\mathbb{L}(\delta(\cdot))$ coincides with the classical Lipschitz space $\mathbb{L}(n/\beta-n/p)$. Moreover, when $\beta=p(x)\equiv p$, the space $\mathbb{\widetilde L}(n/\beta-n/p)$ coincides
with the bounded mean oscillation space $BMO(\rn)$.
\end{remark}

\begin{property}\label{p2.1}
Let $p(\cdot)\in \mathcal{P}^{\log}(\mathbb{R}^{n})$, $0\le\delta(\cdot)<n$, $1<\beta<p_-$, $p(x)\ge p_\infty$ for a.e. $x\in\rn$ and $ \delta(\cdot)/n=1/\beta-1/p(\cdot)$. Then $\mathbb{L}(\delta(\cdot))=\mathbb{\widetilde L}(\delta(\cdot))$ with equivalent norms.
\end{property}

\begin{proof}
Under the assumptions of $p(\cdot)\in \mathcal{P}^{\log}(\mathbb{R}^{n})$, $0\le\delta(\cdot)<n$, $1<\beta<p_-$, $p'(x)\le (p')_\infty$ (implies by $p(x)\ge p_\infty$) for a.e. $x\in\rn$ and $0\leq \delta(\cdot)/n=1/\beta-1/p(\cdot)<1$,
by \cite[(5.4)]{pr17}, we know that $b\in \mathbb{\widetilde L}(\delta(\cdot))\Rightarrow b\in \mathbb{L}(\delta(\cdot))$ and $\|b\|_{\mathbb{L}(\delta(\cdot))}\le C \|b\|_{\mathbb{\widetilde L}(\delta(\cdot))}$.

Conversely, let us prove $b\in \mathbb{L}(\delta(\cdot))\Rightarrow b\in \mathbb{\widetilde L}(\delta(\cdot))$. For any fixed cube $Q$ and any $x\in Q$, we have
\begin{align*}
|b(x)-b_Q|
&\leq \frac{1}{|Q|}\int_Q|b(x)-b(y)|\,dy\\
&\leq \|b\|_{\mathbb{L}(\delta(\cdot))} \frac{1}{|Q|}\int_Q|x-y|^{\delta(x)}\,dy
\leq \|b\|_{\mathbb{L}(\delta(\cdot))}|Q|^{\frac{\delta(x)}{n}}.
\end{align*}
From this, Property \ref{p2.2}(i), $\delta(\cdot)/n=1/\beta-1/p(\cdot)$ and Property \ref{p2.2}(iv), it follows that
\begin{align*}
\int_Q|b(x)-b_Q|\,dx
&\leq \|b\|_{\mathbb{L}(\delta(\cdot))}\int_Q|Q|^{\frac{\delta(x)}{n}}\chi_Q(x)\,dx\\
&\leq \|b\|_{\mathbb{L}(\delta(\cdot))}\left\||Q|^{\frac{\delta(\cdot)}{n}}\right\|_{L^{p(\cdot)}(Q)}\|\chi_Q\|_{L^{p^\prime(\cdot)}(Q)}\\
&=\|b\|_{\mathbb{L}(\delta(\cdot))}\lf\||Q|^{\frac{1}{\beta}-\frac{1}{p(\cdot)}}\r\|_{L^{p(\cdot)}(Q)}\|\chi_Q\|_{L^{p^\prime(\cdot)}(Q)}\\
&=\|b\|_{\mathbb{L}(\delta(\cdot))}|Q|^{\frac{1}{\beta}}\lf\||Q|^{-\frac{1}{p(\cdot)}}\r\|_{L^{p(\cdot)}(Q)}\|\chi_Q\|_{L^{p^\prime(\cdot)}(Q)}\\
&\leq \|b\|_{\mathbb{L}(\delta(\cdot))}|Q|^{\frac{1}{\beta}}\|\chi_Q\|_{L^{p^\prime(\cdot)}(Q)},
\end{align*}
which implies that
\begin{align*}
\frac{1}{|Q|^{\frac{1}
{\beta}}\|\chi_Q\|_{L^{p^\prime(\cdot)}(\mathbb{R}^n)}}\int_Q|b(x)-b_Q|\, dx\leq \|b\|_{\mathbb{L}(\delta(\cdot))}.
\end{align*}
Then by taking the supremum over all cubes $Q$ on both sides of the above inequality, we obtain $b\in \mathbb{\widetilde L}(\delta(\cdot))$ and $\|b\|_{\mathbb{\widetilde L}(\delta(\cdot))}\le \|b\|_{\mathbb{L}(\delta(\cdot))}$.
The proof of Property \ref{p2.1} is completed.
\end{proof}

\section{Some characterizations of Variable Lipschitz Spaces \label{s3}}
This section is devoted to show some characterizations of variable Lipschitz spaces via the boundedness of commutators of $[b,\,\mathcal{M}^\sharp]$, $\mathcal{M}_{\alpha,\,b}$, $[b,\,\mathcal{M}_\alpha]$ $\mathcal{M}_b$ or $[b,\,\mathcal{M}]$ on variable Lebesgue spaces.

\begin{theorem}\label{t3.3}
Let $0<\delta(\cdot)<n$, $p(\cdot), q(\cdot),\,r(\cdot)\in\mathcal{P}^{\log}(\mathbb{R}^{n})$ and $\beta\in(1, p_-)$ satisfying the following assumptions:
\begin{enumerate}
\item[(i)] $\delta_->0$, $(r\delta)_+<n$ and $r_\infty\delta_+<n$;
\item[(ii)] $p(x)\ge p_\infty$ for a.e. $x\in\rn$ and $1/\beta+1/(p')_+>1$;
\item[(iii)]  $1/r(\cdot)-1/q(\cdot)=\delta(\cdot)/n=1/\beta-1/p(\cdot)$.
\end{enumerate}
Then $b\in \mathbb{L}(\delta(\cdot))$ and $b\ge 0$ if and only if $[b,\,\mathcal{M}^\sharp]: L^{r(\cdot)}(\rn)\rightarrow L^{q(\cdot)}(\rn)$.
\end{theorem}

\begin{theorem}\label{t3.4}
Let $0\le\alpha<n$, $0<\delta(\cdot)<n$, $p(\cdot), q(\cdot),\,r(\cdot)\in\mathcal{P}^{\log}(\mathbb{R}^{n})$ and $\beta\in(1, p_-)$ satisfying the following assumptions:
\begin{enumerate}
\item[(i)] $\delta_->0$, $(r(\alpha+\delta))_+<n$ and $r_\infty\lf(\alpha+\delta(\cdot)\r)_+<n$;
\item[(ii)] $p(x)\ge p_\infty$ for a.e. $x\in\rn$;
\item[(iii)]  $1/\beta-1/p(\cdot)=\delta(\cdot)/n=1/r(\cdot)-\alpha/n-1/q(\cdot)$.
\end{enumerate}
Then $b\in \mathbb{L}(\delta(\cdot))$ if and only if $\mathcal{M}_{\alpha,\,b}$: $L^{r(\cdot)}(\rn) \rightarrow L^{q(\cdot)}(\rn)$.
\end{theorem}

\begin{theorem}\label{t3.5}
Let $0\le\alpha<n$, $0<\delta(\cdot)<n$, $p(\cdot), q(\cdot),\,r(\cdot)\in\mathcal{P}^{\log}(\mathbb{R}^{n})$ and $\beta\in(1, p_-)$ satisfying the following assumptions:
\begin{enumerate}
\item[(i)] $\delta_->0$, $(r(\alpha+\delta))_+<n$ and $r_\infty\lf(\alpha+\delta(\cdot)\r)_+<n$;
\item[(ii)] $p(x)\ge p_\infty$ for a.e. $x\in\rn$ and $1/\beta+1/(p')_+>1$;
\item[(iii)]  $1/\beta-1/p(\cdot)=\delta(\cdot)/n=1/r(\cdot)-\alpha/n-1/q(\cdot)$.
\end{enumerate}
Then $b\in \mathbb{L}(\delta(\cdot))$ and $b\ge 0$ if and only if $[b,\,\mathcal{M}_\alpha]$: $L^{r(\cdot)}(\rn) \rightarrow L^{q(\cdot)}(\rn)$.
\end{theorem}

\begin{remark}\label{r3.x1}
There are two remarks for the assumptions of Theorem \ref{t3.3}.
\begin{enumerate}
\item[(i)] The following assumptions are used to guarantee $\mathbb{L}(\delta(\cdot))=\mathbb{\widetilde L}(\delta(\cdot))$:  $p(\cdot)\in \mathcal{P}^{\log}(\mathbb{R}^{n})$, $0<\delta(\cdot)<n$, $1<\beta<p_-$, $p(x)\ge p_\infty$ for a.e. $x\in\rn$ and $\delta(\cdot)/n=1/\beta-1/p(\cdot)$.
\item[(ii)] The following assumptions are used to guarantee variable fractional maximal operator $\mathcal{M}_\delta(\cdot): L^{r(\cdot)}(\rn)\to L^{q(\cdot)}(\rn)$ (see Remark \ref{r2.1} below): $\delta(\cdot)\in (0,n)$ and $r(\cdot)\in\mathcal{P}^{\log}(\rn)$ satisfying $\delta_->0$,
$(r\delta)_+<n$ and $r_\infty\delta_+<n$, and $q(\cdot)$ is defined from $1/q(\cdot)=1/r(\cdot)-\delta(\cdot)/n$.
\end{enumerate}
Similar assumptions are also given in Theorems \ref{t3.4} and \ref{t3.5}.
\end{remark}

By using Theorems \ref{t3.4} and \ref{t3.5} with $\alpha=0$, we may characterizes the variable spaces $\mathbb{L}(\delta(\cdot))$ in terms of the boundedness of operators $\mathcal{M}_b$ and $[b,\, \mathcal{M}]$.

\begin{corollary}\label{t3.1}
Under the assumptions of Theorem \ref{t3.4} with $\alpha=0$ (i.e. $\mathcal{M}_b$:=$\mathcal{M}_{0,\,b}$). Then $b\in \mathbb{L}(\delta(\cdot))$ if and only if $\mathcal{M}_b$: $L^{r(\cdot)}(\rn) \rightarrow L^{q(\cdot)}(\rn)$.
%Let $0< \delta(\cdot) < n$, $ p, r \in\mathcal{P}^{\log}(\mathbb{R}^{n})$ with $p(x)\geq p_{\infty}$ for every $x\in \mathbb{R}^{n}$ and $1<\beta\leq p_{-}$ such that $1/r(x)-1/q(x)=\delta(x)/n=1/\beta-1/p(x)$ with $\sup_{x\in \mathbb{R}^n} r(x)\delta(x)<n$.
%\begin{enumerate}
%\item[(i)] Let $\Omega$ be a bounded open set in $\rn$. If $b\in \mathbb{L}(\delta(\cdot);\Omega)$, then $\mathcal{M}^\Omega_b$: $L^{r(\cdot)}(\Omega)\rightarrow L^{q(\cdot)}(\Omega)$.
%\item[(ii)] If $\mathcal{M}_b$: $L^{r(\cdot)}(\mathbb{R}^n)\rightarrow L^{q(\cdot)}(\mathbb{R}^n)$, then $b\in \mathbb{L}(\delta(\cdot))$.
%\end{enumerate}
\end{corollary}

\begin{corollary}\label{t3.2}
Under the assumptions of Theorem \ref{t3.5} with $\alpha=0$ (i.e. $[b,\,\mathcal{M}]$:=$[b,\,\mathcal{M}_0]$). Then $b\in \mathbb{L}(\delta(\cdot))$ and $b\geq 0$ if and only if $[b,\,\mathcal{M}]$: $L^{r(\cdot)}(\rn) \rightarrow L^{q(\cdot)}(\rn)$.
%Let $0<\delta(\cdot) < n$, $p,\,r \in\mathcal{P}^{\log}(\mathbb{R}^{n})$ with $p(x)\geq p_{\infty}$ for every $x\in \mathbb{R}^{n}$ and $1<\beta\leq p_{-}$ such that $1/r(x)-1/q(x)=\delta(x)/n=1/\beta-1/p(x)$ with $\sup_{x\in \mathbb{R}^n} r(x)\delta(x)<n$.
%\begin{enumerate}
%\item[(i)] Let $\Omega$ be a bounded open set in $\rn$. If $b\in \mathbb{L}(\delta(\cdot);\Omega)$ and $b\geq 0$, then $[b,\,\mathcal{M}_\Omega]$: $L^{r(\cdot)}(\Omega)\rightarrow L^{q(\cdot)}(\Omega)$.
%\item[(ii)] If $[b,\,\mathcal{M}]$: $L^{r(\cdot)}(\mathbb{R}^n)\rightarrow L^{q(\cdot)}(\mathbb{R}^n)$, then
%$b\in \mathbb{L}(\delta(\cdot))$ and $b\geq 0$.
%\end{enumerate}
\end{corollary}

%%%%%%%%%%%%%%%%%%%%%%%%%%%%%%%%%%%%%%%%%%%%%%%%%%%%%%%%%%%%%%%%%%%%%%%%%
%%%%%%%%%%%%%%%%%%%%%%%%%%%  2. Main results   %%%%%%%%%%%%%%%%%%%%%%%%%
%%%%%%%%%%%%%%%%%%%%%%%%%%%%%%%%%%%%%%%%%%%%%%%%%%%%%%%%%%%%%%%%%%%%%%%%%

To prove Theorems \ref{t3.3}, \ref{t3.4} and \ref{t3.5}, we need several technical lemmas as follows.

\begin{lemma}\label{l2.4}
\begin{enumerate}
\item[(i)]\cite[Lemma 2.9]{i10} Let $q(\cdot)\in \mathcal{P}^{\log}(\mathbb{R}^n)$. Then there exists a positive constant $C$ such that
$$\frac{1}{|Q|}\|\chi_Q\|_{L^{q(\cdot)}(\mathbb{R}^n)}\|\chi_Q\|_{L^{q^\prime(\cdot)}(\mathbb{R}^n)}\leq C$$
for all cubes $Q$ in $\mathbb{R}^n$.

\item[(ii)] Let $0\leq\gamma<n$ and $p(\cdot),\, q(\cdot)\in \mathcal{P}^{\log}(\rn)$ such that $1/q(\cdot)=1/p(\cdot)-\gamma/n$ with $p_+<n/\gamma$.
Then there exists a positive constant $C$ such that
$$\|\chi_Q\|_{L^{q(\cdot)}(\mathbb{R}^n)}\|\chi_Q\|_{L^{p^\prime(\cdot)}(\mathbb{R}^n)}\leq C|Q|^{1-\frac{\gamma}{n}}$$
for all cubes $Q$ in $\mathbb{R}^n$.
\end{enumerate}
\end{lemma}

\begin{proof}[Proof of Lemma \ref{l2.4}(ii)]
From (i) and $\frac{1}{(1-\frac{\gamma}{n})q(x)}+\frac{1}{(1-\frac{\gamma}{n})p^\prime(x)}=1$, we deduce that
$$\frac{1}{|Q|}\|\chi_Q\|_{L^{(1-\frac{\gamma}{n})q(\cdot)}(\mathbb{R}^n)}\|\chi_Q\|_{L^{(1-\frac{\gamma}{n})p^\prime(\cdot)}(\mathbb{R}^n)}\leq C.
$$
By this and Property \ref{p2.2}(iii), we obtain
\begin{align*}
&\quad \ \|\chi^{\frac{n}{n-\gamma}}_Q\|_{L^{(1-\frac{\gamma}{n})q(\cdot)}(\mathbb{R}^n)}
\|\chi^{\frac{n}{n-\gamma}}_Q\|_{L^{(1-\frac{\gamma}{n})p^\prime(\cdot)}(\mathbb{R}^n)}
\leq C|Q|\\
&\Longleftrightarrow
\|\chi_Q\|^{\frac{n}{n-\gamma}}_{L^{(1-\frac{\gamma}{n})q(\cdot)\frac{n}{n-\gamma}}(\mathbb{R}^n)}
\|\chi_Q\|^{\frac{n}{n-\gamma}}_{L^{(1-\frac{\gamma}{n})p^\prime(\cdot)\frac{n}{n-\gamma}}(\mathbb{R}^n)}\leq C|Q|\\
&\Longleftrightarrow\|\chi_Q\|^{\frac{n}{n-\gamma}}_{L^{q(\cdot)}(\mathbb{R}^n)}\|\chi_Q\|^{\frac{n}{n-\gamma}}_{L^{p^\prime(\cdot)}(\mathbb{R}^n)}\leq C|Q|\\
&\Longleftrightarrow\|\chi_Q\|_{L^{q(\cdot)}(\mathbb{R}^n)}\|\chi_Q\|_{L^{p^\prime(\cdot)}(\mathbb{R}^n)}\leq C|Q|^{1-\frac{\gamma}{n}}.
\end{align*}
The proof of Lemma \ref{l2.4}(ii) is completed.
\end{proof}
Next, we recall some notion and lemmas on variable fractional maximal function (see, for example, \cite{ms12, ks03}). Let $0<\delta(\cdot)<n$. The {\it potential operator} $I_{\delta(\cdot)}$ of  locally integrable function $f$ is defined by
$$I_{\delta(x)}f(x):=\int_\rn \frac{f(y)}{|x-y|^{n-\delta(x)}}\,dy.$$
The {\it variable fractional maximal operator} $\mathcal{M}_{\delta(\cdot)}$ of locally integrable function $f$ is defined by
\begin{align}\label{e2.13}
\mathcal{M}_{\delta(x)}(f)(x) := \sup_{Q\ni x}\frac{1}{|Q|^{1-\frac{\delta(x)}{n}}}\int_Q|f(y)|\,dy,
\end{align}
where the supremum is taken over all cubes $Q\subset\rn$ containing $x$.

The following lemma shows the boundedness of potential operators $I_{\delta(\cdot)}$ on variable Lebesgue spaces, which comes from \cite[Theorem 1.2]{ms12} with the measurable function space $\mathcal{L}_{p,\,q,\,\beta}(\rn)$ in it choosing a special case $\mathcal{L}_{p(\cdot),\,q,\,\beta}(\rn)=\mathcal{L}_{p(\cdot),\,0,\,0}(\rn)=L^{p(\cdot)}(\rn)$.

\begin{lemma}\label{l2.9}
Let $\delta(\cdot)\in (0,n)$ satisfying $\delta_->0$ and let $r(\cdot)\in\mathcal{P}^{\log}(\rn)$ satisfying
$(r\delta)_+<n$ and $r_\infty\delta_+<n$.
Define $q(\cdot)$ as $1/q(\cdot)=1/r(\cdot)-\delta(\cdot)/n$.
Then the potential operator $I_{\delta(\cdot)}$ is bounded from $L^{r(\cdot)}(\rn)$ into $L^{q(\cdot)}(\rn)$.
\end{lemma}

\begin{remark}
Under the assumptions of Lemma \ref{l2.9} and from the pointwise estimate $$(\mathcal{M}_{\delta(x)}f)(x)\leq C(I_{\delta(x)}|f|)(x)$$ (see \label{r2.1}\cite[P.\,909]{ks03}), where the constant $C$ does not depend on $f$, we deduce that the maximal operator $\mathcal{M}_{\delta(\cdot)}$ is bounded from $L^{r(\cdot)}(\rn)$ into $L^{q(\cdot)}(\rn)$.
\end{remark}

%For a function $b$ defined on $\mathbb{R}^n$, we denote by
%\begin{equation}\label{e2.011}
%b^-(x) :=\left\{\begin{array}{ll}
%0 & \qquad b(x)\geq 0\\
%|b(x)| &\qquad  b(x) < 0\\
%\end{array} \right.
%\end{equation}
%and $b^+(x):=|b(x)|-b^-(x)$. Obviously, $b^+(x)-b^-(x)=b(x)$.

\begin{lemma}\label{l3.4}
Let $p(\cdot)\in \mathcal{P}^{\log}(\rn)$  and $Q\subset\rn$ is a cube satisfying $|Q|\le C_0$ where $C_0>0$ is a sufficient small constant depends on $p_+$. Then $\|\chi_Q\|_{L^{p(\cdot)}(\rn)}\le |Q|^{1/p_+}$.
\end{lemma}
\begin{proof}
For any $\lambda\in(0,\,1)$, if the side length of $Q$ is sufficient small, we have
\begin{align*}
1\ge \frac{|Q|^{p_+}}{\lambda}=\int_\rn \lf(\frac{\chi_Q}{\lambda}\r)^{p_+}dx\ge\int_\rn \lf(\frac{\chi_Q}{\lambda}\r)^{p(x)}dx.
\end{align*}
Then by the definition of $\|\cdot\|_{L^{p(\cdot)}(\rn)}$ and pick $\lambda=|Q|^{1/p_+}$, we obtain $$\|\chi_Q\|_{L^{p(\cdot)}(\rn)}\le |Q|^{1/p_+},$$ which completes the proof of Lemma \ref{l3.4}.
\end{proof}

\begin{proof}[Proof of Theorem \ref{t3.3}]
(i) For any $b\in \mathbb{L}(\delta(\cdot))$ and $b\geq 0$, let us show
\begin{equation}\label{e3.x4}
[b,\,\mathcal{M}^\sharp]: L^{r(\cdot)}(\rn)\rightarrow L^{q(\cdot)}(\rn)
\end{equation}
holds true. Let $f\in L^{r(\cdot)}(\rn)$. From Definition \ref{d2.4}(i) and \eqref{e2.13}, we deduce that, for any $x\in\rn$,
\begin{align*}
&|[b,\,\mathcal{M}^\sharp]f(x)|\\
&=\lf|\sup_{Q\ni x}\frac{b(x)}{|Q|}\int_Q|f(y)-f_Q|dy-\sup_{Q\ni x}\frac{1}{|Q|}\int_Q|b(y)f(y)-(bf)_Q|\,dy\r|\\\nonumber
&\leq \sup_{Q\ni x}{\frac{1}{|Q|}\int_Q|b(y)-b(x)||f(y)|+|b(x)f_Q-(bf)_Q|\,dy}\\\nonumber
&\leq \|b\|_{\mathbb{L}(\delta(\cdot))}\sup_{Q\ni x}\frac{1}{|Q|}\int_Q|y-x|^{\delta(x)}|f(y)|\,dy\\ \nonumber
&\quad+\sup_{Q\ni x}\lf|\frac{b(x)}{|Q|}\int_Qf(z)dz-\frac{1}{|Q|}\int_Qb(z)f(z)\,dz\r|\\\nonumber
&\leq \|b\|_{\mathbb{L}(\delta(\cdot))}\sup_{Q\ni x}\frac{1}{|Q|}\int_Q|y-x|^{\delta(x)}|f(y)|\,dy+\sup_{Q\ni x}\frac{1}{|Q|}\int_Q|b(x)-b(z)||f(z)|\,dz\\\nonumber
&\leq \|b\|_{\mathbb{L}(\delta(\cdot))}\sup_{Q\ni x}\frac{1}{|Q|}\int_Q|y-x|^{\delta(x)}|f(y)|\,dy+\|b\|_{\mathbb{L}(\delta(\cdot))}\sup_{Q\ni x}\frac{1}{|Q|}\int_Q|x-z|^{\delta(x)}|f(z)|\,dz\\\nonumber
&\leq C\|b\|_{\mathbb{L}(\delta(\cdot))}\sup_{Q\ni x}\frac{1}{|Q|^{1-\frac{\delta(x)}{n}}}\int_Q|f(y)|\,dy\\\nonumber
&\leq C\|b\|_{\mathbb{L}(\delta(\cdot))}\mathcal{M}_{\delta(x)}f(x).\nonumber
\end{align*}
Then, by this and $\mathcal{M}_{\delta(\cdot)}: L^{r(\cdot)}(\rn)\to L^{q(\cdot)}(\rn)$ (see Remark \ref{r2.1}) with the assumptions (i) and (iii) of Theorem \ref{t3.3}, we obtain
$$\|[b,\,\mathcal{M}^\sharp](f)\|_{L^{q(\cdot)}(\rn)}\leq C||b||_{\mathbb{L}(\delta(\cdot))}\|f\|_{L^{r(\cdot)}(\rn)}$$
and hence \eqref{e3.x4} holds true.

(ii) Conversely, suppose that $[b,\,\mathcal{M}^\sharp]$: $L^{r(\cdot)}(\mathbb{R}^n)\rightarrow L^{q(\cdot)}(\mathbb{R}^n)$, let us prove \begin{equation}\label{e3.x5}
b\in \mathbb{L}(\delta(\cdot))\quad {\rm and}\quad b\ge 0.
\end{equation}
For any fixed cube $Q$, repeating the proof of \cite[P.\,3333]{bmr}, we have
\begin{align*}
\mathcal{M}^\sharp(\chi_Q)(x)=\frac{1}{2} \quad {\rm for} \ {\rm all}\quad x\in Q.
\end{align*}
By this and \eqref{e2.08}, we obtain
\begin{align*}
\lf\|\lf(b-2\mathcal{M}^\sharp(b\chi_Q)\r)\chi_Q\r\|_{L^{q(\cdot)}(\mathbb{R}^n)}
&=\lf\|2\lf(\frac{1}{2}b-\mathcal{M}^\sharp(b\chi_Q)\r)\chi_Q\r\|_{L^{q(\cdot)}(\mathbb{R}^n)}\\
&=\lf\|2\lf(b\mathcal{M}^\sharp(\chi_Q)-\mathcal{M}^\sharp(b\chi_Q)\r)\chi_Q\r\|_{L^{q(\cdot)}(\mathbb{R}^n)}\\
&= 2\|[b,\,\mathcal{M}^\sharp](\chi_Q)\|_{L^{q(\cdot)}(\mathbb{R}^n)}\\
&\leq C\|\chi_Q\|_{L^{r(\cdot)}(\mathbb{R}^n)},
\end{align*}
which implies that
\begin{equation}\label{e3.9}
\frac{\lf\|\lf(b-2\mathcal{M}^\sharp(b\chi_Q)\r)\chi_Q\r\|_{L^{q(\cdot)}(\mathbb{R}^n)}}{\|\chi_Q\|_{L^{r(\cdot)}(\mathbb{R}^n)}}
\leq C.
\end{equation}

For any fixed cube $Q\subset \mathbb{R}^n$, we have
\begin{equation}\label{e3.10}
|b_Q|\leq 2\mathcal{M}^\sharp(b\chi_Q)(x) \quad {\rm for\ any} \quad x\in Q. \quad \mathrm{(see\ \cite[(2)]{bmr})}
\end{equation}
Let $E:=\{x\in Q:b(x)\leq b_Q\}$. Notice that $b_Q$ is the mean value of $\int_Q b(y)\,dy$, we obtain
\begin{align}\label{e4.0}
\int_E|b(x)-b_Q|\,dx=\int_{Q\setminus E}|b(x)-b_Q|\,dx.
\end{align}
Moreover, for any $x\in E$, we have $b(x)\leq b_Q\leq |b_Q|\leq 2\mathcal{M}^\sharp(b\chi_Q)(x)$. Then
\begin{align}\label{e4.1}
|b(x)-b_Q|\leq |b(x)-2\mathcal{M}^\sharp(b\chi_Q)(x)| \quad {\rm for\ any} \quad x\in E.
\end{align}
Since $1/q^\prime(\cdot)=1/p^\prime(\cdot)+[1/\beta-1/r(\cdot)]$, from \eqref{e4.0}, \eqref{e4.1}, (i) and (ii) of Property \ref{p2.2}, and \eqref{e3.9}, we conclude that
\begin{align}\label{e3.x1}
&\frac{1}{|Q|^{\frac{1}{\beta}}\|\chi_Q\|_{L^{p^\prime(\cdot)}(\mathbb{R}^n)}}\int_Q|b(x)-b_Q|\,dx \nonumber \\\nonumber
&\quad=\frac{2}{|Q|^{\frac{1}{\beta}}\|\chi_Q\|_{L^{p^\prime(\cdot)}(\mathbb{R}^n)}}\int_E|b(x)-b_Q|\,dx\\
&\quad\leq \frac{2}{|Q|^{\frac{1}{\beta}}\|\chi_Q\|_{L^{p^\prime(\cdot)}(\mathbb{R}^n)}}
\int_Q|b(x)-2\mathcal{M}^\sharp(b\chi_Q)(x)|\,dx\\\nonumber
&\quad\leq \frac{C}{|Q|^{\frac{1}{\beta}}\|\chi_Q\|_{L^{p^\prime(\cdot)}(\mathbb{R}^n)}}
\lf\|\lf(b-2\mathcal{M}^\sharp(b\chi_Q)\r)\chi_Q\r\|_{L^{q(\cdot)}(\mathbb{R}^n)}
\|\chi_Q\|_{L^{q^\prime(\cdot)}(\mathbb{R}^n)}\\\nonumber
&\quad\leq \frac{C}{|Q|^{\frac{1}{\beta}}\|\chi_Q\|_{L^{p^\prime(\cdot)}(\mathbb{R}^n)}}
\lf\|\lf(b-2\mathcal{M}^\sharp(b\chi_Q)\r)\chi_Q\r\|_{L^{q(\cdot)}(\mathbb{R}^n)}
\|\chi_Q\|_{L^{p^\prime(\cdot)}(\mathbb{R}^n)}
\|\chi_Q\|_{L^{(\frac{1}{\beta}-\frac{1}{r(\cdot)})^{-1}}(\mathbb{R}^n)}\\\nonumber
&\quad=\frac{C}{|Q|^{\frac{1}{\beta}}}\frac{\lf\|\lf(b-2\mathcal{M}^\sharp(b\chi_Q)\r)\chi_Q\r\|_{L^{q(\cdot)}(\mathbb{R}^n)}}
{\|\chi_Q\|_{L^{r(\cdot)}(\mathbb{R}^n)}}
\|\chi_Q\|_{L^{r(\cdot)}(\mathbb{R}^n)}\|\chi_Q\|_{L^{(\frac{1}{\beta}-\frac{1}{r(\cdot)})^{-1}}(\mathbb{R}^n)}\\\nonumber
&\quad\le\frac{C}{|Q|^{\frac{1}{\beta}}}
\|\chi_Q\|_{L^{r(\cdot)}(\mathbb{R}^n)}\|\chi_Q\|_{L^{(\frac{1}{\beta}-\frac{1}{r(\cdot)})^{-1}}(\mathbb{R}^n)}.
\end{align}
Let $h^\prime(\cdot):=(1/\beta-1/r(\cdot))^{-1}$. Then we obtain
$$\frac{1}{r(\cdot)}+\frac{1}{h^\prime(\cdot)}=1-\lf(1-\frac{1}{\beta}\r)\quad
{\rm and}\quad h(\cdot)=\frac{1}{1-\frac{1}{h^\prime(\cdot)}}=\frac{1}{1-\frac{1}{\beta}+\frac{1}{r(\cdot)}},$$
which implies that $$h_+=\frac{1}{1-\frac{1}{\beta}+\frac{1}{r_+}}<\frac{1}{1-\frac{1}{\beta}}.$$
Therefore, by Lemma \ref{l2.4}(ii), we have
$$\|\chi_Q\|_{L^{r(\cdot)}(\mathbb{R}^n)}\|\chi_Q\|_{L^{(\frac{1}{\beta}-\frac{1}{r(\cdot)})^{-1}}(\mathbb{R}^n)}\leq C|Q|^{\frac{1}{\beta}}.$$
By this and \eqref{e3.x1}, we obtain
$$\frac{1}{|Q|^{\frac{1}{\beta}}\|\chi_Q\|_{L^{p^\prime(\cdot)}(\mathbb{R}^n)}}\int_Q|b(x)-b_Q|\,dx\leq C.$$
This shows that $b\in\mathbb{\widetilde L}(\delta(\cdot))$ and
hence by Property \ref{p2.1} with the assumptions (ii) and (iii) of Theorem \ref{t3.3}, $b\in\mathbb{L}(\delta(\cdot))$.

Next, let us prove $b\geq 0$. It also suffices to show $b^-=0$, where $b^-:=-\min\{b,0\}$ and $b^+:=|b|-b^-$. For any $x\in Q$, by \eqref{e3.10}, we have
$$2\mathcal{M}^\sharp(b\chi_Q)(x)-b(x)\geq |b_Q|-b(x)=|b_Q|-b^+(x)+b^-(x).$$
From this, we deduce that
\begin{align}\label{e3.11}
\frac{1}{|Q|}\int_Q|2\mathcal{M}^\sharp(b\chi_Q)(x)-b(x)|\,dx
&\geq \frac{1}{|Q|}\int_Q(|b_Q|-b^+(x)+b^-(x))\,dx\\\nonumber
&=|b_Q|-\frac{1}{|Q|}\int_Qb^+(x)dx+\frac{1}{|Q|}\int_Qb^-(x)\,dx.\nonumber
\end{align}
On the other hand, by \eqref{e3.x1}, we have
\begin{align*}
\frac{1}{|Q|^{\frac{1}{\beta}+1}\|\chi_Q\|_{L^{p^\prime(\cdot)}(\mathbb{R}^n)}}\int_Q|2\mathcal{M}^\sharp(b\chi_Q)(x)-b(x)|\,dx
%&\quad\leq \frac{C}{|Q|^{\frac{1}{\beta}+1}\|\chi_Q\|_{L^{p^\prime(\cdot)}(\mathbb{R}^n)}}\lf\|\lf(2\mathcal{M}^\sharp(b\chi_Q)-b\r)\chi_Q\r\|_{L^{q(\cdot)}(\mathbb{R}^n)}
%\|\chi_Q\|_{L^{q^\prime(\cdot)}(\mathbb{R}^n)}\\
%&\quad\leq \frac{C}{|Q|^{\frac{1}{\beta}+1}\|\chi_Q\|_{L^{p^\prime(\cdot)}(\mathbb{R}^n)}}
%\lf\|\lf(2\mathcal{M}^\sharp(b\chi_Q)-b\r)\chi_Q\r\|_{L^{q(\cdot)}(\mathbb{R}^n)} \|\chi_Q\|_{L^{p^\prime(\cdot)}(\mathbb{R}^n)}\|\chi_Q\|_{L^{(\frac{1}{\beta}-\frac{1}{r(\cdot)})^{-1}}
%(\mathbb{R}^n)}\\
%&\quad\leq  \frac{C}{|Q|^{\frac{1}{\beta}+1}}\frac{\lf\|\lf(2\mathcal{M}^\sharp(b\chi_Q)-b\r)\chi_Q\r\|_{L^{q(\cdot)}(\mathbb{R}^n)}}
%{\|\chi_Q\|_{L^{r(\cdot)}(\mathbb{R}^n)}}
%\|\chi_Q\|_{L^{r(\cdot)}(\mathbb{R}^n)}\|\chi_Q\|_{L^{(\frac{1}{\beta}-\frac{1}{r(\cdot)})^{-1}}
%(\mathbb{R}^n)}
\le C|Q|^{-1}.
\end{align*}
This, together with \eqref{e3.11}, gives
\begin{align}\label{e3.12}
\frac{1}{|Q|^{\frac{1}{\beta}}\|\chi_Q\|_{L^{p^\prime(\cdot)}(\mathbb{R}^n)}}\lf(|b_Q|-\frac{1}{|Q|}\int_Qb^+(x)dx+\frac{1}{|Q|}\int_Qb^-(x)\,dx\r)
\leq C|Q|^{-1}.
\end{align}

By letting the side length of $Q$ sufficient small and using \eqref{e3.12} and Lemma \ref{l3.4}, we obtain
$$|b_Q|-\frac{1}{|Q|}\int_Qb^+(x)dx+\frac{1}{|Q|}\int_Qb^-(x)\,dx\leq C|Q|^{\frac{1}{\beta}+\frac{1}{(p^\prime)_+}-1}.$$
Let the side length of $Q$ tends to 0 (then $|Q|\rightarrow 0$) with $Q\ni x$, Legesgue's differentation theorem assures that the limit of the left-hand side of \eqref{e3.12} equals to
$$|b(x)|-b^+(x)+b^-(x)=2b^-(x)=2|b^-(x)|.$$
And the right-hand side of tends to 0 due to $1/\beta+1/(p')_+>1$. So, we have $b^-=0$ and hence \eqref{e3.x5} holds true. We finish the proof of Theorem \ref{t3.3}.
\end{proof}

\begin{proof}[Proof of Theorem \ref{t3.4}]
%To prove Theorem \ref{t3.4}, we borrow some ideas from the proof of \cite[Theorem 1.3]{zsw19}.
(i) For any $b\in \mathbb{L}(\delta(\cdot))$, let us show
\begin{equation}\label{e3.x6}
\mathcal{M}_{\alpha,\,b}: L^{r(\cdot)}(\rn) \rightarrow L^{q(\cdot)}(\rn)
\end{equation}
holds true. Let $f\in L^{r(\cdot)}(\rn)$ and  $[\alpha+\delta(\cdot)]/n=1/r(\cdot)-1/q(\cdot)$. From Definition \ref{d2.4}(i) and \eqref{e2.13}, we
deduce that, for any $x\in\rn$,
\begin{align*}
|\mathcal{M}_{\alpha,\,b}(f)(x)|
&=\sup_{Q\ni x}\frac{1}{|Q|^{1-\frac{\alpha}{n}}}\int_Q|b(x)-b(y)||f(y)|\,dy\\\nonumber
&\leq \|b\|_{\mathbb{L}(\delta(\cdot))}\sup_{Q\ni x}\frac{1}{|Q|^{1-\frac{\alpha}{n}}}\int_Q|x-y|^{\delta(x)}|f(y)|\,dy\\\nonumber
&\leq \|b\|_{\mathbb{L}(\delta(\cdot))}\sup_{Q\ni x}\frac{1}{|Q|^{1-\frac{\alpha+\delta(x)}{n}}}\int_Q|f(y)|\,dy\\\nonumber
&\leq \|b\|_{\mathbb{L}(\delta(\cdot))}\mathcal{M}_{\alpha+\delta(x)}(f)(x).\nonumber
\end{align*}
Then, by this and
$\mathcal{M}_{\alpha+\delta(\cdot)}: L^{r(\cdot)}(\rn)\to L^{q(\cdot)}(\rn)$ (see Remark \ref{r2.1}) with the assumptions (i) and (iii) of Theorem \ref{t3.4}, we obtain
$$\|\mathcal{M}_{\alpha,\,b}(f)\|_{L^{q(\cdot)}(\rn)}\leq C\|b\|_{\mathbb{L}(\delta(\cdot))}\|\mathcal{M}_{\alpha+\delta(\cdot)}(f)\|_{L^{q(\cdot)}(\rn)}\leq C\|b\|_{\mathbb{L}(\delta(\cdot))}\|f\|_{L^{r(\cdot)}(\rn)}$$
and hence \eqref{e3.x6} holds true.

(ii) Conversely, suppose that $\mathcal{M}_{\alpha,\,b}$: $L^{r(\cdot)}(\mathbb{R}^n)\rightarrow L^{q(\cdot)}(\mathbb{R}^n)$, let us prove
$b\in \mathbb{L}(\delta(\cdot))$.

For any fixed cube $Q$, noting that for all $x\in Q$, we have,
%\begin{align*}
%|b(x)-b_Q|
%&\leq \frac{1}{|Q|}\int_Q|b(x)-b(y)|\,dy\\
%&=\frac{1}{|Q|}\int_Q|b(x)-b(y)|\chi_Q(y)\,dy\\
%&\leq |Q|^{-\frac{\alpha}{n}}\mathcal{M}_{\alpha,\,b}(\chi_Q)(x).
%\end{align*}
%Then, by this, for all $x\in \mathbb{R}^n$, we have
\begin{align*}
|(b(x)-b_Q)\chi_Q(x)|\leq |Q|^{-\frac{\alpha}{n}}\mathcal{M}_{\alpha,\,b}(\chi_Q)(x).
\end{align*}
By this and that $\mathcal{M}_{\alpha,\,b}$ is bounded from $L^{r(\cdot)}(\mathbb{R}^n)$ to $L^{q(\cdot)}(\mathbb{R}^n)$, we obtain
\begin{align*}
\|(b-b_Q)\chi_Q\|_{L^{q(\cdot)}(\mathbb{R}^n)}
\leq |Q|^{-\frac{\alpha}{n}}\|\mathcal{M}_{\alpha,\,b}(\chi_Q)\|_{L^{q(\cdot)}(\mathbb{R}^n)}
\leq C|Q|^{-\frac{\alpha}{n}}\|\chi_Q\|_{L^{r(\cdot)}(\mathbb{R}^n)},
\end{align*}
which implies that
\begin{equation}\label{e3.14}
\frac{\|(b-b_Q)\chi_Q\|_{L^{q(\cdot)}(\mathbb{R}^n)}}{\|\chi_Q\|_{L^{r(\cdot)}(\mathbb{R}^n)}}\leq C|Q|^{-\frac{\alpha}{n}}.
\end{equation}
Since $1/q^\prime(\cdot)=1/p^\prime(\cdot)+[1/\beta+\alpha/n-1/r(\cdot)]$, by (i) and (ii) of Property \ref{p2.2} and \eqref{e3.14}, we conclude that
\begin{align}\label{e3.x2}
&\frac{1}{|Q|^{\frac{1}{\beta}}\|\chi_Q\|_{L^{p^\prime(\cdot)}(\mathbb{R}^{n})}}\int_Q|b(x)-b_Q|\,dx\\\nonumber
%&=\frac{1}{|Q|^{\frac{1}{\beta}}\|\chi_Q\|_{L^{p^\prime(\cdot)}(\mathbb{R}^{n})}}\int_Q|b(x)-b_Q|\chi_Q(x)\,dx\\\nonumber
&\quad\leq \frac{C}{|Q|^{\frac{1}{\beta}}\|\chi_Q\|_{L^{p^\prime(\cdot)}(\mathbb{R}^{n})}}\|(b-b_Q)\chi_Q\|_{L^{q(\cdot)}(\mathbb{R}^n)}
\|\chi_Q\|_{L^{q^\prime(\cdot)}(\mathbb{R}^n)}\\\nonumber
&\quad\leq \frac{C}{|Q|^{\frac{1}{\beta}}\|\chi_Q\|_{L^{p^\prime(\cdot)}(\mathbb{R}^{n})}}\|(b-b_Q)\chi_Q\|_{L^{q(\cdot)}(\mathbb{R}^n)}
\|\chi_Q\|_{L^{p^\prime(\cdot)}(\mathbb{R}^n)}\|\chi_Q\|_{L^{(\frac{1}{\beta}+\frac{\alpha}{n}-\frac{1}{r(\cdot)})^{-1}}(\mathbb{R}^n)}\\\nonumber
&\quad=\frac{C}{|Q|^{\frac{1}{\beta}}}\frac{\|(b-b_Q)\chi_Q\|_{L^{q(\cdot)}(\mathbb{R}^n)}}{\|\chi_Q\|_{L^{r(\cdot)}(\mathbb{R}^n)}}
\|\chi_Q\|_{L^{r(\cdot)}(\mathbb{R}^n)}\|\chi_Q\|_{L^{(\frac{1}{\beta}+\frac{\alpha}{n}-\frac{1}{r(\cdot)})^{-1}}(\mathbb{R}^n)}\\\nonumber
&\quad\leq \frac{C}{|Q|^{\frac{1}{\beta}+\frac{\alpha}{n}}}\|\chi_Q\|_{L^{r(\cdot)}(\mathbb{R}^n)}
\|\chi_Q\|_{L^{(\frac{1}{\beta}+\frac{\alpha}{n}-\frac{1}{r(\cdot)})^{-1}}(\mathbb{R}^n)}.
\end{align}
Let $h^\prime(\cdot):=[1/\beta+\alpha/n-1/r(\cdot)]^{-1}$.
 Then we obtain $$\frac{1}{r(\cdot)}+\frac{1}{h^\prime(\cdot)}=1-\lf[1-\lf(\frac{1}{\beta}+\frac{\alpha}{n}\r)\r]\quad {\rm and}\quad h(\cdot)=\frac{1}{1-\frac{1}{h^\prime(\cdot)}}=\frac{1}{1-\frac{1}{\beta}-\frac{\alpha}{n}+\frac{1}{r(\cdot)}},$$ which implies that $$h_+=\frac{1}{1-\frac{1}{\beta}-\frac{\alpha}{n}+\frac{1}{r_+}}<\frac{1}{1-\frac{1}{\beta}-\frac{\alpha}{n}}.$$
Therefore, by Lemma \ref{l2.4}(ii), we obtain
\begin{align}\label{e3.x3}
\|\chi_Q\|_{L^{r(\cdot)}(\mathbb{R}^n)}\|\chi_Q\|_{L^{(\frac{1}{\beta}+\frac{\alpha}{n}-\frac{1}{r(\cdot)})^{-1}}(\mathbb{R}^n)}\leq C|Q|^{\frac{1}{\beta}+\frac{\alpha}{n}}.
\end{align}
By this and \eqref{e3.x2}, we obtain
$$\frac{1}{|Q|^{\frac{1}{\beta}}\|\chi_Q\|_{L^{p^\prime(\cdot)}(\mathbb{R}^{n})}}\int_Q|b(x)-b_Q|\,dx\leq C.$$
This shows that $b\in\mathbb{\widetilde L}(\delta(\cdot))$ and
hence by Property \ref{p2.1} with the assumptions (ii) and (iii) of Theorem \ref{t3.4}, $b\in\mathbb{L}(\delta(\cdot))$. We finish the proof
of Theorem \ref{t3.4}.
\end{proof}

Let $\gamma\geq0$. For a fixed cube $Q_0$, the {\it fractional maximal operator $\mathcal{M}_{\gamma,\,Q_0}$} with respect to $Q_0$ of a locally integrable function $f$ is defined by
$$\mathcal{M}_{\gamma,\,Q_0}(f)(x):=\sup_{Q\ni x,\,Q\subseteq Q_0}\frac{1}{|Q|^{1-\frac{\gamma}{n}}}\int_Q|f(y)|\,dy,
\ \ x\in\rn,$$
where the supremum is taken over all cubes $Q$ such that $x\in Q\subseteq Q_0$.
When $\gamma=0$, we simply write $\mathcal{M}_{Q_0}$ instead of $\mathcal{M}_{0,\,Q_0}$.

\begin{proof}[Proof of Theorem \ref{t3.5}]
%To prove Theorem \ref{t3.5}, we borrow some ideas from the proof of \cite[Theorem 1.1]{zsw19}.

(i) For any $b\in \mathbb{L}(\delta(\cdot))$ and $b\geq 0$, let us show
\begin{equation}\label{e3.x7}
[b,\,\mathcal{M}_\alpha]: L^{r(\cdot)}(\rn)\rightarrow L^{q(\cdot)}(\rn)
\end{equation}
holds true.
Let $f\in L^{r(\cdot)}(\rn)$ and $[\alpha+\delta(\cdot)]/n=1/r(\cdot)-1/q(\cdot)$. For $x\in\rn$, we have
\begin{align*}
|[b,\,\mathcal{M}_\alpha](f)(x)|
&=\lf|\sup_{Q\ni x}\frac{1}{|Q|^{1-\frac{\alpha}{n}}}\int_Q b(x)|f(y)|dy-\sup_{Q\ni x}\frac{1}{|Q|^{1-\frac{\alpha}{n}}}\int_Q |b(y)f(y)|\,dy\r|\\\nonumber
&\leq \sup_{Q\ni x}\frac{1}{|Q|^{1-\frac{\alpha}{n}}}\int_Q |b(x)-b(y)||f(y)|\,dy
=\mathcal{M}_{\alpha,\,b}(f)(x).\nonumber
\end{align*}
By this and Theorem \ref{t3.4}, we obtain
$$ \|[b,\,\mathcal{M}_\alpha](f)\|_{L^{q(\cdot)}(\rn)}\le \|\mathcal{M}_{\alpha,\,b}(f)\|_{L^{q(\cdot)}(\rn)}\leq C||b||_{\mathbb{L}(\delta(\cdot))}\|f\|_{L^{r(\cdot)}(\rn)}$$
and hence \eqref{e3.x7} holds true.

(ii) Conversely, suppose that $[b,\,\mathcal{M}_\alpha]$: $L^{r(\cdot)}(\mathbb{R}^n)\rightarrow L^{q(\cdot)}(\mathbb{R}^n)$, let us prove \begin{equation}\label{e3.x8}
b\in \mathbb{L}(\delta(\cdot))\quad {\rm and}\quad b\ge 0.
\end{equation}

For any fixed cube $Q$, noting that for all $x\in Q$, we have
$$\mathcal{M}_\alpha(\chi_Q)(x)=\mathcal{M}_{\alpha,\,Q}(\chi_Q)(x)=|Q|^{\frac{\alpha}{n}} \quad {\rm and} \quad \mathcal{M}_\alpha(b\chi_Q)(x)=\mathcal{M}_{\alpha,\,Q}(b)(x) \ (\mathrm{see}\ \cite[(2.4)]{zw09}).$$
Then, for any $x\in Q$, from this and \eqref{e2.10}, we deduce that
\begin{align*}
b(x)-|Q|^{-\frac{\alpha}{n}}\mathcal{M}_{\alpha,\,Q}(b)(x)
&=|Q|^{-\frac{\alpha}{n}}\lf[b(x)|Q|^{\frac{\alpha}{n}}-\mathcal{M}_{\alpha,\,Q}(b)(x)\r]\\
&=|Q|^{-\frac{\alpha}{n}}\lf[b(x)\mathcal{M}_\alpha(\chi_Q)(x)-\mathcal{M}_\alpha(b\chi_Q)(x)\r]\\
&=|Q|^{-\frac{\alpha}{n}}[b,\,\mathcal{M}_\alpha](\chi_Q)(x),
\end{align*}
which implies that
$$\lf(b(x)-|Q|^{-\frac{\alpha}{n}}\mathcal{M}_{\alpha,\,Q}(b)(x)\r)\chi_Q(x)=|Q|^{-\frac{\alpha}{n}}[b,\,\mathcal{M}_\alpha](\chi_Q)(x)\chi_Q(x).$$
By this and $[b,\,\mathcal{M}_\alpha]:\,L^{r(\cdot)}(\rn)\to L^{q(\cdot)}(\rn)$, we obtain
\begin{align*}
\lf\|\lf(b-|Q|^{-\frac{\alpha}{n}}\mathcal{M}_{\alpha,\,Q}(b)\r)\chi_Q\r\|_{L^{q(\cdot)}(\mathbb{R}^n)}
%&\leq |Q|^{-\frac{\alpha}{n}}\|[b,\,\mathcal{M}_\alpha](\chi_Q)\|_{L^{q(\cdot)}(\mathbb{R}^n)}\\
\leq C|Q|^{-\frac{\alpha}{n}}\|\chi_Q\|_{L^{r(\cdot)}(\mathbb{R}^n)},
\end{align*}
which implies that
\begin{equation}\label{e3.16}
\frac{\lf\|\lf(b-|Q|^{-\frac{\alpha}{n}}\mathcal{M}_{\alpha,\,Q}(b)\r)\chi_Q\r\|_{L^{q(\cdot)}(\mathbb{R}^n)}}{\|\chi_Q\|_{L^{r(\cdot)}(\mathbb{R}^n)}}\leq
C|Q|^{-\frac{\alpha}{n}}.
\end{equation}
Now, let us show $b\in \mathbb{L}(\delta(\cdot))$. By Property \ref{p2.1}, we only need to show $b\in \mathbb{\widetilde L}(\delta(\cdot))$. For any fixed cube $Q$, let $E:=\{x\in Q: b(x)\leq b_Q\}$.
For $x\in E$, we have
$$b(x)\leq b_Q\leq |b_Q|\leq |Q|^{-\frac{\alpha}{n}}\mathcal{M}_{\alpha,\,Q}(b)(x).$$
Thus,
\begin{align}\label{e4.4}
|b(x)-b_Q|\leq \lf|b(x)-|Q|^{-\frac{\alpha}{n}}\mathcal{M}_{\alpha,\,Q}(b)(x)\r|.
\end{align}
Then, by the mean value property of $b_Q$, \eqref{e4.4}, (i) and (ii) of Property \ref{p2.2} with $1/q^\prime(\cdot)=1/p^\prime(\cdot)+(1/\beta+\alpha/n-1/r(\cdot))$, \eqref{e3.16} and \eqref{e3.x3}, we conclude that
\begin{align}\label{e4.3}
&\frac{1}{|Q|^{\frac{1}{\beta}}\|\chi_Q\|_{L^{p^\prime(\cdot)}(\mathbb{R}^n)}}\int_Q|b(x)-b_Q|\,dx\\\nonumber
%&=\frac{1}{|Q|^{\frac{1}{\beta}}\|\chi_Q\|_{L^{p^\prime(\cdot)}(\mathbb{R}^n)}}\int_{E\bigcup(Q\setminus E)}|b(x)-b_Q|\,dx\\\nonumber
&\hs=\frac{2}{|Q|^{\frac{1}{\beta}}\|\chi_Q\|_{L^{p^\prime(\cdot)}(\mathbb{R}^n)}}\int_E|b(x)-b_Q|\,dx\\\nonumber
%&\leq \frac{2}{|Q|^{\frac{1}{\beta}}\|\chi_Q\|_{L^{p^\prime(\cdot)}(\mathbb{R}^n)}}\int_E|b(x)-|Q|^{-\frac{\alpha}{n}}\mathcal{M}_{\alpha,\,Q}(b)(x)|\,dx\\\nonumber
%&\leq \frac{2}{|Q|^{\frac{1}{\beta}}\|\chi_Q\|_{L^{p^\prime(\cdot)}(\mathbb{R}^n)}}\int_Q|b(x)-|Q|^{-\frac{\alpha}{n}}\mathcal{M}_{\alpha,\,Q}(b)(x)|\,dx\\\nonumber
&\hs\leq\frac{2}{|Q|^{\frac{1}{\beta}}\|\chi_Q\|_{L^{p^\prime(\cdot)}
(\mathbb{R}^n)}}\int_Q|b(x)-|Q|^{-\frac{\alpha}{n}}
\mathcal{M}_{\alpha,\,Q}(b)(x)|\chi_Q(x)\,dx\\\nonumber
&\hs\leq \frac{C}{|Q|^{\frac{1}{\beta}}\|\chi_Q\|_{L^{p^\prime(\cdot)}
(\mathbb{R}^n)}}\|(b(x)-|Q|^{-\frac{\alpha}{n}}
\mathcal{M}_{\alpha,\,Q}(b)(x))\chi_Q\|_{L^{q(\cdot)}(\mathbb{R}^n)}
\|\chi_Q\|_{L^{q^\prime(\cdot)}(\mathbb{R}^n)}\\\nonumber
&\hs\leq \frac{C}{|Q|^{\frac{1}{\beta}}\|\chi_Q\|_{L^{p^\prime(\cdot)}
(\mathbb{R}^n)}}\|(b(x)-|Q|^{-\frac{\alpha}{n}}
\mathcal{M}_{\alpha,\,Q}(b)(x))\chi_Q\|_{L^{q(\cdot)}(\mathbb{R}^n)}\\ \nonumber
&\hs\hs\times\|\chi_Q\|_{L^{p^\prime(\cdot)}(\mathbb{R}^n)}\|\chi_Q\|_{L^{(\frac{1}{\beta}
+\frac{\alpha}{n}-\frac{1}{r(\cdot)})^{-1}}(\mathbb{R}^n)}
\\\nonumber
&\hs\leq\frac{C}{|Q|^{\frac{1}{\beta}}}\frac{\lf\|\lf(b-|Q|^{-\frac{\alpha}{n}}
\mathcal{M}_{\alpha,\,Q}(b)\r)\chi_Q\r\|_{L^{q(\cdot)}(\mathbb{R}^n)}}
{\|\chi_Q\|_{L^{r(\cdot)}(\mathbb{R}^n)}} \|\chi_Q\|_{L^{r(\cdot)}(\mathbb{R}^n)}\|\chi_Q\|_{L^{(\frac{1}{\beta}+\frac{\alpha}{n}-\frac{1}{r(\cdot)})^{-1}}(\mathbb{R}^n)}\\\nonumber
&\hs\leq \frac{C}{|Q|^{\frac{1}{\beta}}}\cdot |Q|^{-\frac{\alpha}{n}}\cdot |Q|^{\frac{1}{\beta}+\frac{\alpha}{n}}=C,
\end{align}
which implies that $b\in\mathbb{\widetilde L}(\delta(\cdot))$ and hence by Property \ref{p2.1} with the assumptions (ii) and (iii) of Theorem \ref{t3.5}, $b\in \mathbb{L}(\delta(\cdot))$.

Next, let us prove $b\geq 0$. To do this, it suffices to show $b^-=0$, where $b^-:=-\min\{b,0\}$. Let $b^+:=|b|-b^-$, then $b=b^+-b^-$. For any fixed cube $Q$,
$$0\leq b^+(x)\leq |b(x)|\leq |Q|^{-\frac{\alpha}{n}}\mathcal{M}_{\alpha,\,Q}(b)(x),\quad x\in Q.$$
Therefore, for $x\in Q$, we have
$$0\leq b^-(x)\leq |Q|^{-\frac{\alpha}{n}}\mathcal{M}_{\alpha,\,Q}(b)(x)-b^+(x)+b^-(x)=|Q|^{-\frac{\alpha}{n}}\mathcal{M}_{\alpha,\,Q}(b)(x)-b(x).$$
Then, for any cube $Q$, it follows from \eqref{e4.3} that
\begin{align}\label{e3.4}
\frac{1}{|Q|^{\frac{1}{\beta}}\|\chi_Q\|_{L^{p^\prime(\cdot)}(\mathbb{R}^n)}}\int_Q b^-(x)\,dx\leq \frac{1}{|Q|^{\frac{1}{\beta}}\|\chi_Q\|_{L^{p^\prime(\cdot)}(\mathbb{R}^n)}}\int_Q||Q|^{-\frac{\alpha}{n}}\mathcal{M}_{\alpha,\,Q}(b)(x)-b(x)|\,dx\leq C.
\end{align}
By this, letting the side length of $Q$ sufficient small and using Lemma \ref{l3.4}, we obtain
$$\frac1{|Q|}\int_Qb^-(x)\,dx\leq C|Q|^{\frac{1}{\beta}+\frac{1}{(p^\prime)_+}-1}.$$
Thus, $b^-=0$ due to $1/\beta+1/(p')_+>1$ and  Lebesgue's differentiation theorem.
Finally, we finish the proof of \eqref{e3.x8} and hence the proof
of Theorem \ref{t3.5}.
\end{proof}

\section*{Funding and/or Conflicts of interests/Competing interests}
This project is supported by NSFC (Nos.\,11861062,\,12161083), the Natural Science Foundation of Xinjiang Uyghur Autonomous Region (No.\,2020D01C048)  and  Xinjiang key laboratory of applied mathematics (No.\,XJDX1401).

The authors declare that there is no conflict of interests regarding the publication of this paper.

\bigskip\medskip

\noindent Xuechun Yang, Zhenzhen Yang and Baode Li (Corresponding author),
\medskip

\noindent College of Mathematics and System Sciences\\
 Xinjiang University\\
 Urumqi, 830017\\
P. R. China
\smallskip

\noindent{E-mail }:\\
\texttt{2760978447@qq.com} (Xuechun Yang)\\
\texttt{1756948251@qq.com} (Zhenzhen Yang)\\
\texttt{baodeli@xju.edu.cn} (Baode Li)\\
\bigskip \medskip

\end{document}